\documentclass[11pt]{article}

\usepackage{amssymb, amsmath, amsthm, mathtools, comment}

\newcommand{\dd}{\mathrm{d}}
\newcommand{\E}{\mathbf{E}}
\newcommand{\p}{\mathbf{P}}

\newcommand{\Z}{\mathbb{Z}}

\theoremstyle{plain}
\newtheorem{theorem}{Theorem}
\newtheorem{lemma}[theorem]{Lemma}

\newtheorem{corollary}[theorem]{Corollary}

\theoremstyle{definition}
\newtheorem*{example}{Example}

\theoremstyle{remark}

\title{Heavy-tailed critical Galton--Watson processes with immigration}
\author{P\'eter Kevei\footnote{Bolyai Institute, University of Szeged, Szeged,
Hungary. \\
E-mail: \texttt{kevei@math.u-szeged.hu}} \and 
Kata Kubatovics\footnote{Bolyai Institute, University of Szeged, Szeged, Hungary. \\
E-mail: \texttt{kubatovics@server.math.u-szeged.hu}}}

\date{}

\begin{document}

\maketitle

\begin{abstract}
Consider a critical Galton--Watson branching process with immigration, where 
the offspring distribution belongs to the domain of attraction of a $(1 + \alpha)$-stable 
law with $\alpha \in (0,1)$, and the immigration distribution
either (i) has finite mean, or (ii) belongs to the domain of attraction of a 
$\beta$-stable law with $\beta \in (\alpha, 1)$. 
We show that the tail of the stationary distribution is regularly varying.  
We analyze the stationary process, determine its tail process, and establish a stable central limit theorem for the partial sums. The norming sequence is different from the one corresponding to the tail of the stationary law. In particular, the extremal index of the process is $0$. \smallskip

\noindent 
\textit{Keywords:} Galton--Watson process with immigration, critical offspring 
distribution, regularly varying stationary sequence, heavy-tailed time series,
tail behavior \\
\noindent \textit{MSC2010:} {60J80, 60F05}
\end{abstract}

\section{Introduction}

Let $X_0 = 0$, 
and for $n \geq 0$
\begin{equation*} \label{eq:X-def}
X_{n+1} = \sum_{i=1}^{X_n} A_{n+1,i} + B_{n+1} =: 
\theta_{n+1} \circ X_n + B_{n+1},
\end{equation*}
where $\{A_{n,i}: n \geq 1, i \geq 1 \}$ are nonnegative integer-valued 
iid (independent, identically distributed) random variables, and 
independently of the $A$'s, $\{ B_n : n \geq 1 \}$
is an iid sequence of nonnegative integer-valued random variables, and $A$
and $B$ are the corresponding generic copies. 
Then $A_{n+1,i}$ is the number of offspring of individual $i$ in 
generation $n$, and $B_{n+1}$ is the number of immigrants. 
Put 
\begin{equation*} \label{eq:f-g-def}
f(s) = \E (s^A), \quad g(s) = \E (s^B), \quad s \in [0,1],
\end{equation*}
for the generating function of the offspring and the immigration 
distribution, respectively.
Foster and Williamson \cite[Theorem (iii)]{FosterWill} proved that
$X_n \stackrel{\mathcal{D}}{\to} X_\infty$ for some 
finite random variable $X_\infty$, where $\stackrel{\mathcal{D}}{\to}$
stands for convergence in distribution, if and only if 
\begin{equation} \label{eq:FW-cond}
\int_{0}^1 \frac{1-g(s)}{f(s) - s} \dd s < \infty.
\end{equation}
In this case the law of $X_\infty$ is the unique stationary distribution of 
the Markov chain, see Corollary \ref{cor:X-inf}.
Condition \eqref{eq:FW-cond} holds in the subcritical case ($\E A < 1$) 
if and only if $\E \log B < \infty$, this was proved 
by Quine \cite{Quine} in the multitype setup. However, \eqref{eq:FW-cond} 
can also hold in the critical case  ($\E A = 1$), when necessarily 
$\E A^2 = \infty$, see \cite{FosterWill}.

In the subcritical case, the tail behavior of the stationary law and 
the properties of the stationary process are well-studied.
If either the offspring or the immigration distribution is regularly 
varying (with certain index range), then the tail behavior of the 
stationary law, and the properties of the stationary process 
were investigated by Basrak et al.~\cite{BKP}. More recently, 
Foss and Miyazawa \cite{FossMiy} considered not only regularly 
varying tails, but more general tail behavior for the stationary 
law. In the random environment setup, Basrak and Kevei \cite{BK},
and Kevei \cite{K24} obtained conditions ensuring the regular 
variation of the stationary law, and also investigated the 
properties of the stationary process.

Much less is known in the critical case. Recently, Guo and Hong 
\cite{GuoHong} proved that if both the offspring and the immigration
law is regularly varying then (under additional conditions on the 
slowly varying functions) the stationary law is also regularly varying.
Zhao \cite{Zhao} obtained tail asymptotics of the 
stationary law in the borderline case, 
when the offspring distribution has tail 
$\p ( A > x ) \sim [x (\log x)^{1 + c}]^{-1}$,
$x \to \infty$, with $c > 0$, and the immigration has tail 
$\p ( B > x ) \sim x^{-1}$, $x \to \infty$. 
The properties of the stationary process are not studied.
Here, and later on $\sim$
stands for asymptotic equality, i.e.~$a(x) \sim b(x)$ as 
$x \to L \in \{0,1,\infty\}$,
if $\lim_{x \to L} \tfrac{a(x)}{b(x)} = 1$. The same convention applies to
sequences.

In the present paper we assume that the process is critical and, 
for some $\alpha \in (0,1)$ and slowly varying function $\ell_A$,
\begin{equation} \label{eq:f-ass}
f(s) = s + (1-s)^{1+\alpha} \ell_A (1/(1-s)).
\end{equation}
On the immigration we assume that one of the following 
conditions holds: 
\begin{itemize}
\item[(B1)]  $\E B < \infty$, or
 
\item[(B2)] for some $\beta \in (\alpha, 1)$ and slowly varying 
function $\ell_B$,
\begin{equation} \label{eq:g-ass}
g(s) :=  \E s^B = 1 - (1-s)^\beta \ell_B ( 1 / (1-s)).
\end{equation}
\end{itemize}
It is easy to check that under these assumptions \eqref{eq:FW-cond} holds,
thus the stationary distribution exists.

Although there are not many results for processes with immigration, 
the critical branching process without immigration, where the offspring distribution 
has the form \eqref{eq:f-ass}, has attracted a lot of attention. 
In 1968, Slack \cite{Slack} determined the extinction rate 
of the process, and obtained a Yaglom-type limit theorem in this setup. 
Borovkov and Vatutin \cite{BorVat} showed that the tail of the maximum 
of the process is $c x^{-1}$, where the tail-index $\alpha$ in \eqref{eq:f-ass}
only appears in the multiplicative constant $c$. In the same setup,
Fleischmann et al.~\cite{VWF}, among other results, 
obtained tail asymptotics for the total population.

Extending the results of Guo and Hong \cite{GuoHong}, 
in Theorem \ref{thm:tail} we show 
that the tail of the stationary distribution is regularly varying. 
We further investigate the properties of the stationary process. The theory 
of heavy-tailed time series has attracted considerable mathematical attention
recently, see e.g.~the monographs by Mikosch and Wintenberger \cite{MikoschWinten}, and by 
Kulik and Soulier \cite{KS19}.
In Theorem \ref{thm:tail-proc} we determine the tail process of the 
stationary chain. The tail process, introduced by Basrak and Segers \cite{BasrakSegers} reveals the behavior of the chain around its 
extremes. It turns out that the forward tail process is $U_0 (1,1,\ldots)$,
with $U_0$ having Pareto distribution. This rather exotic behavior implies 
that the usual anti-clustering condition does not hold. 
The general theory of heavy-tailed time series works under mild conditions
and implies point process convergence and convergence of partial sums and maxima.
These mild conditions include the anti-clustering condition together with some
mixing property. As a consequence, we cannot use the general theory.
In Theorem \ref{thm:sum} we prove a stable central limit theorem for the partial
sums. Comparing the tail of the stationary law and the norming sequence, 
we obtain that the extremal index of the process is 0.
Stationary time series with extremal index 0 are considered pathological. Apart from the Lindley process in queueing theory,
we provide a further natural example of such a process.

Section \ref{sect:stat-dist} contains the results on the stationary 
distribution. In Section \ref{sect:proc} we investigate the stationary 
process. All the proofs are gathered together in Section \ref{sect:proofs}.

\section{Stationary distribution} \label{sect:stat-dist}

If the offspring generating function has the form  \eqref{eq:f-ass} 
with some $\alpha \in (0,1)$ and slowly varying function $\ell_A$, 
and either (B1) or (B2) holds for the immigration generating function,
then the Foster--Williamson condition \eqref{eq:FW-cond} holds, thus 
$X_n \stackrel{\mathcal{D}}{\to} X_\infty$, in particular a
stationary distribution exists. 
From basic Markov chain theory we deduce that the stationary 
distribution is unique. For definitions on Markov chains we refer to 
Douc et al.~\cite{Douc}. 
Let 
\begin{equation*} \label{eq:def-k0}
k_0 = \min \{ k: \, \p ( B = k) > 0 ) \}.
\end{equation*}
Then $k_0$ is an accessible atom for the process $(X_n)$. Indeed,
for any $k> 0$
\[
\p ( X_{n+1} = k_0 | X_n = k ) = (\p ( A=0))^k \p ( B = k_0 ) > 0.
\]
Therefore, $(X_n)$ is irreducible. Furthermore, as stationary 
distribution exists, $(X_n)$ is recurrent, with a unique stationary 
distribution, see e.g.~Theorems 7.1.4 and 7.2.1 in \cite{Douc}.
Let $f_n$ denote the $n$-fold composition of the generating function $f$,
that is $f_0(s) = s$, and $f_{n+1}(s) = f_n(f(s))$, $n \geq 1$. Then,
by conditioning
\[
\E s^{X_{n+1}} = g(s) \E (f(s)^{X_n}) = 
g(s) g(f(s)) \E (f_2(s)^{X_{n-1}}) = \prod_{i=0}^n g(f_i(s)).
\]
We summarize this as follows.

\begin{corollary} \label{cor:X-inf}
If the offspring generating function has the form  \eqref{eq:f-ass} 
with some $\alpha \in (0,1)$ and slowly varying function $\ell_A$, 
and either (B1) or (B2) holds for the immigration generating function,
then a unique stationary distribution exists. Furthermore,
$X_n \stackrel{\mathcal{D}}{\rightarrow} X_\infty$, where $X_\infty$ 
has the stationary distribution, whose generating function is given by
\begin{equation} \label{eq:def-phi}
\varphi(s) : = \E s^{X_\infty} = \prod_{i=1}^\infty g(f_n(s)), \quad 
s \in [0,1].
\end{equation}
\end{corollary}

Analyzing the asymptotics of the generating function in \eqref{eq:def-phi}, 
we determine the tail behavior of $X_\infty$.
For easier reference, we state a well-known result on the relation of the 
tail asymptotics at infinity and the generating function asymptotics at 1. 
We could not locate the precise statement in the literature, but it
follows easily from a Tauberian theorem, see Corollary 8.1.7 in Bingham
et al.~\cite{BGT}, or Section 1.1 in \cite{GuoHong}.

\begin{lemma} \label{lemma:tail-genfunc}
Let $Y$ be a nonnegative integer-valued random variable, $\mu \in (0,1)$, and
$\ell$ a slowly varying function. The following are equivalent:
\begin{itemize}
\item[(i)] $\p ( Y > x) \sim \frac{\ell(x)}{x^\mu \Gamma(1-\mu)}$ as $x \to \infty$;
\item[(ii)] $1 - \E ( s^{Y}) \sim \ell(1/(1-s)) (1-s)^{\mu}$ as $s \uparrow 1$.
\end{itemize}
\end{lemma}

We can state our main result on the tail of the stationary distribution.

\begin{theorem} \label{thm:tail}
Assume that with some $\alpha \in (0,1)$, and with a slowly varying 
function $\ell_A$ condition \eqref{eq:f-ass} holds. 
\begin{itemize}
\item[(i)] If (B1) holds, then
\[
\p ( X_\infty > x ) \sim 
\frac{g'(1)}{(1-\alpha) \Gamma(\alpha)}
x^{-(1-\alpha)} \ell_A(x)^{-1} \quad \text{as } x \to \infty.
\]

\item[(ii)] If (B2) holds, then 
\[
\p ( X_\infty > x ) \sim 
\frac{1}{(\beta - \alpha) \Gamma(1 - \beta + \alpha)} x^{-(\beta-\alpha)} 
\frac{\ell_B(x)}{\ell_A(x)}
\quad \text{as } x \to \infty.
\]
\end{itemize}
\end{theorem}

In case (ii), under further assumption on the slowly varying functions, 
Guo and Hong \cite[Theorem 1.1]{GuoHong} obtained the regular variation 
of the tail of
the stationary distribution, without specifying the slowly varying function.
In our result we do not assume further assumptions on the slowly varying functions,
and we obtain explicit expression for the tail. Case (i) is completely new.

We also note that the result can be extended further to 
the boundary cases, when $\alpha = 0$ or $1$, or when $\beta = \alpha$ or $1$.
In these boundary cases however, further assumptions are needed on the 
slowly varying functions in order to deduce the tail asymptotic. In particular,
the corresponding slowly varying functions have to belong to the de Haan class,
see \cite[Corollary 8.1.7]{BGT} and the remark after it. See also our remarks
after Lemmas \ref{lemma:T-tail} and \ref{lemma:rsum-tail}.
In order to keep the presentation simpler, we decided to exclude these cases.

\begin{example}[Power-fractional offspring distribution or theta-branching]
For a Galton--Watson process without immigration, the generating function of 
the population in generation $n$ is the $n$-fold composition of the 
offspring generating function, which usually cannot be calculated 
explicitly. The linear fractional offspring distribution is an 
important example, because the $n$-fold composition has an 
explicit form, see e.g.~Section I.4 in Athreya and Ney \cite{AthreyaNey}. 
The linear fractional distribution is a modified geometric distribution,
and it has finite exponential moments. More recently, Sagitov and Lindo
\cite{SagitovLindo} introduced a new class of possibly defective 
offspring distributions, including heavy-tailed distributions, 
where the composition is explicit. The resulting process is 
called \emph{theta-branching process}. 
Alsmeyer and Hoang \cite{Alsmeyer} call this class 
\emph{power-fractional distributions}. 

A critical power-fractional generating function has the form
\[
f(s) = 1 - \left[ (1-s)^{-\alpha} + 1 \right]^{-1/\alpha}, \quad 
s \in [0,1],
\]
where $\alpha \in (0,1)$.
Then $f$ has the form \eqref{eq:f-ass} with $\ell_A(x) \sim \alpha$ as 
$x \to \infty$. Let $f_n$ stand for the $n$-fold composition. 
Then \cite[formula (1.7)]{Alsmeyer}
\[
f_n(s) = 1  - \frac{1 - s}{(1 + n (1-s)^\alpha)^{1/\alpha}}.
\]
Assume that the immigration is constant 1, i.e.~$g(s) = s$. Then,
by Corollary \ref{cor:X-inf} the generating function of the 
stationary distribution is
\begin{equation} \label{eq:phi-pf1}
\varphi(s) = \prod_{n=0}^\infty 
\left( 1  - \frac{1 - s}{(1 + n (1-s)^\alpha)^{1/\alpha}} \right).
\end{equation}
If $\alpha = 1/2$ the above formula can be made explicit. Indeed,
with $t = \sqrt{1-s}$, we have 
\[
1  - \frac{1 - s}{(1 + n (1-s)^{1/2})^{2}}
= 1 - \frac{t^2}{(1 + n t)^2} = \frac{(1 + (n-1)t)(1 +(n+1)t)}{(1+nt)^2}.
\]
Thus, substituting back into \eqref{eq:phi-pf1}
\[
\varphi(s) = \prod_{n=0}^\infty 
\frac{(1 + (n-1)t)(1 +(n+1)t)}{(1+nt)^2} = 
1 - t = 1 - \sqrt{1-s}.
\]
Note that $\varphi(s^2)$ is the generating function of the first
return time to 0 in a simple symmetric random walk, see e.g.~Feller 
\cite[Section XI.3]{Feller1}.
\end{example}

\section{The stationary Markov chain} \label{sect:proc}

Let $(X_n)_{n \in \Z}$ be the stationary process version of the 
Galton--Watson process with immigration, defined as
\begin{equation} \label{eq:X-def-stat}
X_{n+1} = \sum_{i=1}^{X_n} A_{n+1,i} + B_{n+1} = \theta_{n+1} \circ X_n + B_{n+1}.
\end{equation}
In what follows we only need that the stationary distribution is 
regularly varying. Assume that for some $\gamma \in (0,1)$ the tail 
of the stationary distribution $X_\infty$ satisfies
\begin{equation} \label{eq:stat-ass}
\p ( X_\infty > x ) \sim \frac{\ell(x)}{x^{\gamma}}, \quad x \to \infty,
\end{equation}
where $\ell$ is a slowly varying function. 
We showed that $ \gamma = 1-\alpha$ if 
\eqref{eq:f-ass} holds and $\E B < \infty$, and  
$\gamma = \beta- \alpha$ if \eqref{eq:f-ass} and \eqref{eq:g-ass}
hold with $1 \geq \beta > \alpha > 0$.

\subsection{Tail process}

The tail process of a stationary process $(X_n)_{n \in \Z}$, 
if it exists, is a process $(U_n)_{n \in \Z}$ such that
for any $k, \ell \geq 0$, as $x \to \infty$
\[
\mathcal{L}  \left( x^{-1} (X_{-\ell}, X_{1-\ell}, \ldots, X_k) | X_0 > x \right) 
\stackrel{\mathcal{D}}{\longrightarrow} (U_{-\ell}, U_{-\ell + 1}, \ldots, U_k),
\]
and $\p (U_0 > y) = y^{-\gamma}$, $y \geq 1$, for some $\gamma > 0$.

The tail process was introduced by Basrak and Segers \cite{BasrakSegers},
and became an essential tool in the analysis of heavy-tailed time series.
It describes the behavior of the process $(X_n)$ around its extreme values.
For definition and properties we refer to the original paper \cite{BasrakSegers},
and to the monographs \cite[Section 4.2]{MikoschWinten} and 
\cite[Chapter 5]{KS19}.
The existence of the tail process is equivalent to the regular variation of 
the stationary process $(X_n)$. Furthermore, it is enough to prove the existence 
of one-sided tail process, that is for $k=0, \ell \geq 0$, or for 
$k \geq 0, \ell = 0$, see \cite[Theorem 4.2.1]{MikoschWinten}.

\begin{theorem} \label{thm:tail-proc}
Let $(X_n)$ be the stationary Markov chain in \eqref{eq:X-def-stat}, and 
assume that for the stationary distribution \eqref{eq:stat-ass} holds.
Then 
\[
\mathcal{L} ( x^{-1} \, (X_i)_{i\geq 0} | X_0 > x ) \to 
U_0 (1,1,1,\ldots),
\]
where $\p ( U_0 > x ) = x^{-\gamma}$, for $x \geq 1$.
\end{theorem}

We note here that the Lindley process has the same tail 
process, see \cite[Section 5.7]{MikoschWinten}.

For a  stationary regularly varying time series $(Y_n)$
choose $a_n$ such that $\p ( Y_0 > a_n ) \sim n^{-1}$. The anti-clustering 
condition $\mathcal{AC}(r_n, a_n)$ (\cite[Condition 4.1]{BasrakSegers},
\cite[Definition 6.1.2]{KS19}, \cite[(6.1.2)]{MikoschWinten}) holds, 
if for each $u > 0$
\begin{equation} \label{eq:antic}
\lim_{m \to \infty} \limsup_{n \to \infty} \p 
\left( \max_{m \leq |k| \leq r_n} |Y_k| > a_n u  \, \big| \, |Y_0| > a_n u\right) 
= 0,
\end{equation}
for some $r_n$, with $r_n \to \infty$ and $r_n / n \to 0$ as $n \to \infty$.
Theorem \ref{thm:tail-proc} combined with
Proposition 4.2 in \cite{BasrakSegers} (\cite[Lemma 6.11]{KS19}) implies that \eqref{eq:antic} does 
not hold for $(X_n)$.
To apply the powerful machinery of heavy-tailed time series one 
needs the anti-clustering condition together with some weak form of mixing.
Therefore, we cannot apply the general theory.

\subsection{Critical Galton--Watson process without immigration}

Here we recall some known results on critical Galton--Watson processes.
Assume that for the offspring generating function
\eqref{eq:f-ass} holds for some $\alpha \in [0,1]$.
For $\alpha = 0$ we further assume that 
$\lim_{x \to \infty} \ell_A(x) = 0$, which implies that 
$\E A = 1$.
Consider a Galton--Watson process with offspring 
generating function $f$ as follows. Let $Z_0 = 1$, and for $n \geq 0$
\begin{equation*} \label{eq:Z-def}
Z_{n+1} = \sum_{i=1}^{Z_n} A_{n+1, i},
\end{equation*}
where $\{ A_{n,i}: n\geq 1, i \geq 1 \}$ are iid random variables with 
generating function $f$.
The process is critical, therefore the total progeny is a.s.~finite,
\begin{equation} \label{eq:T-def}
T = Z_0 + Z_1 + \ldots .
\end{equation}
Let $h(s) = \E s^T$ denote the generating function of $T$. Then
the functional equation
$h(s) = s f(h(s))$ easily implies tail asymptotics for $T$.
The following statement is Lemma 6 (or Theorem 2) in 
\cite{VWF},
see also \cite[Lemma 2.1]{GuoHong}.
In the current form, we allow also $\alpha = 0$, which was 
excluded both in \cite{VWF} and \cite{GuoHong},
and keep a precise track of the appearing slowly varying function.
For completeness, we provide a short proof.
First recall the notion of de Bruijn conjugate, see e.g.~\cite[Theorem 1.5.13]{BGT}. 
If $\ell$ is a slowly varying function, there exists a slowly varying 
function $\ell^{\#}$, unique up to asymptotic equivalence, such that 
\begin{equation*} \label{eq:deBruijn}
\ell(x) \ell^{\#}(x \ell(x)) \to 1, \quad 
\ell^{\#}(x) \ell(x \ell^{\#}(x)) \to 1, \quad x \to \infty.
\end{equation*}
Furthermore, $\ell^{\# \#} \sim \ell$. In what follows, we often use these 
properties.

\begin{lemma} \label{lemma:T-tail}
Assume \eqref{eq:f-ass} with $\alpha \in [0,1]$, and for 
$\alpha = 0$ assume further that $\lim_{x \to \infty}\ell_A(x) = 0$.
Let $\ell_{A,1} (x) = \ell_A(x)^{-1/(1+\alpha)}$.
Then for the generating function of the total progeny, we have
as $s \uparrow 1$,
\begin{equation} \label{eq:h-form}
1- h(s) \sim (1-s)^{1/(1+\alpha)} 
\frac{1}{\ell_{A,1}^{\#}((1-s)^{-1/(1+\alpha)})}.
\end{equation} 
Moreover, for $\alpha > 0$
\begin{equation} \label{eq:T-tail}
\p ( T > x ) \sim 
\left( 
x^{1/(1+\alpha)} 
{\ell_{A,1}^{\#}(x^{1/(1+\alpha)})} \Gamma(\alpha/(1+\alpha) \right)^{-1}.
\end{equation}
\end{lemma}

We note that for $\alpha = 0$ the generating function asymptotics 
\eqref{eq:h-form} does not imply \eqref{eq:T-tail}, only the regular 
variation of the truncated mean. For $\alpha = 0$ the regular variation 
of $\p ( T > x)$ is equivalent to that $1/\ell_{A,1}^{\#}$ belongs 
to the de Haan class $\Pi$, see \cite[Corollary 8.1.7]{BGT} and 
the remark after it.

\subsection{Partial sum}

We are interested in the asymptotic behavior of the 
partial sums $S_n = \sum_{i=1}^n X_i$.

\begin{theorem} \label{thm:sum}
Let $(X_n)$ be the stationary Markov chain in \eqref{eq:X-def-stat}, and 
assume that the critical offspring distribution has generating function 
\eqref{eq:f-ass} for some $\alpha \in (0,1]$, and for the immigration either 
(B1) or (B2) holds.
Let $\eta = 1/(1 + \alpha)$ in case of (B1), 
and $\eta = \beta/(1 + \alpha)$ in case of (B2).
Then
\[
\frac{S_n}{\widetilde \ell(n) n^{1/\eta}}
\stackrel{\mathcal{D}}{\longrightarrow} V(\eta),
\]
where $V(\eta)$ is a nonnegative stable law with index $\eta$, and 
$\widetilde \ell$ is slowly varying.
\end{theorem}

Note that for $\alpha < 1$ the tail of the stationary distribution is 
regularly varying with index $-(1-\alpha)$ in 
case of (B1), and $-(\beta- \alpha)$ in case of (B2), which suggests
larger scaling than in the theorem above.
In particular, this also implies that the stable central limit theorem 
Theorem 9.2.1 in \cite{MikoschWinten}
(see also \cite[Theorem 8.3.1]{KS19}) does not hold, hence either
the mixing or the anti-clustering condition (different from \eqref{eq:antic})
is not satisfied.

\subsection{Maxima}

For a stationary time series extremal dependence can be measured by 
the \emph{extremal index}, introduced by Leadbetter \cite{Leadbetter83},
defined as follows, see Definition 6.1.7 in \cite{MikoschWinten}, or 
\cite[Definition 7.5.1]{KS19}. 
If for each $\tau > 0$ there exists a sequence $u_n(\tau)$ for which 
\begin{equation} \label{eq:def-un}
\lim_{n \to \infty} n \p ( Y_0 > u_n(\tau) ) = \tau,
\end{equation}
such that  
\begin{equation} \label{eq:def-theta}
\lim_{n \to \infty} \p ( M_n \leq u_n(\tau) ) = e^{-\theta_Y \tau},
\end{equation}
for some $\theta_Y \in [0,1]$, where $M_n = \max \{ Y_1, \ldots, Y_n \}$
stands for the partial maxima,
then $\theta_Y$ is the extremal index of $(Y_n)$. 
For an iid sequence $\theta_Y  =1$. For further properties we refer 
to \cite[Section 3.7]{Leadbetter83} and \cite[Section 6]{MikoschWinten}.
\smallskip

Turning back to our stationary Markov chain $(X_n)$, let us assume that 
\eqref{eq:stat-ass} holds.  Then the sequence $u_n(\tau)$ in \eqref{eq:def-un}
can be expressed as
\[
u_n(\tau) \sim (n/\tau)^{1/\gamma} \ell_1^{\#}(n)^{1/\gamma},
\]
where $\ell_1(x) = 1/\ell(x^{1/\gamma})$.
Simply, $M_n \leq S_n$.
By Theorems \ref{thm:tail} and \ref{thm:sum},
in case (i) $\gamma = 1-\alpha$ and $\eta = 1/(1+\alpha)$,
while in case (ii) $\gamma = \beta - \alpha$ and $\eta = \beta/(1+\alpha)$. In both cases 
$u_n(\tau) / n^{1/\eta} \widetilde \ell(n) \to \infty$, implying 
that \eqref{eq:def-theta} holds with $\theta_X  =0$.

\begin{corollary} \label{corr:ext}
Let $(X_n)$ be the stationary Markov chain in \eqref{eq:X-def-stat}.
Under the conditions of Theorem \ref{thm:tail}
the extremal index of $(X_n)$ exists and $\theta_X = 0$.
\end{corollary}

The extremal index $\theta_X  = 0$ is somewhat pathological, see 
Remark 6.1.10 in \cite{MikoschWinten}, and the discussion in \cite[Section 3.7]{Leadbetter83}. There are only a few examples of such process, see e.g.~\cite{Denzel}.
A notable one is the Lindley process in queueing theory, see 
\cite[Section 6.3.6]{MikoschWinten}. Thus we see that the 
extremal behavior of $(X_n)$ and the Lindley process is similar,
as both have the same tail process and both have extremal index 
0. For an intuitive meaning of $\theta_X =0$, see the remark after 
Proposition 6.3.14 in \cite{MikoschWinten}.

Finally, we note that we expect that there is a proper normalization
for the maxima $M_n$. We showed in the beginning of Section \ref{sect:stat-dist} that the Markov chain has an accessible atom,
therefore the idea of the proof of \cite[Proposition 6.3.14]{MikoschWinten} could work. However, determining the 
tail behavior of the cycle maxima remains for further studies.

\section{Proofs} \label{sect:proofs}

\subsection{Proof of Theorem \ref{thm:tail}}

To ease notation, put 
$\ell_1(x)  = \ell_A(x^{1/\alpha}) ^{-1}$.
By Lemma~2 in \cite{Slack},
\[
(1-f_n(0))^\alpha \ell_A 
\left(\frac{1}{1-f_n(0)}\right) \sim \frac{1}{\alpha n}.
\]
Rewriting, we have
\[
(1- f_n(0))^{-\alpha} \ell_1((1-f_n(0))^{-\alpha}) \sim \alpha n,
\]
therefore,
\begin{equation}\label{eq:fn-asy}
1-f_n(0) \sim (\alpha n \ell_1^\#(\alpha n))^{-\tfrac{1}{\alpha}}    
\sim (\alpha n \ell_1^\#(n))^{-\tfrac{1}{\alpha}}.
\end{equation}
Since $f_n(0) \uparrow 1$, for all $s \in (f_1(0),1)$ there exists $m(s)$ such that
\[
f_{m(s)}(0) < s < f_{m(s)+1}(0).
\]
Hence,
\[
1-f_{m(s)+1}(0) < 1-s < 1-f_{m(s)}(0),
\]
and by \eqref{eq:fn-asy}, as $s \uparrow 1$
\begin{equation} \label{eq:m-asy0}
1-s \sim \left(\alpha m(s) \ell_1^\#(m(s)) \right)^{-1/\alpha}.
\end{equation}
It follows that
\begin{equation} \label{eq:m-asy}
m(s) \sim \alpha^{-1} (1-s)^{-\alpha} \, \ell_1\left((1-s)^{-\alpha}\right),
\quad s \uparrow 1.
\end{equation}
Furthermore, by \eqref{eq:fn-asy} as $n + m(s) \to \infty$
\begin{equation*}\label{eq:f_n}
\begin{split}
1 - f_n(s) & \sim 1 - f_n(f_{m(s)}(0)) 
= 1 - f_{n+m(s)}(0) \\ 
& \sim (\alpha(n+m(s)))^{-1/\alpha} \, \ell_1^{\#}(n+m(s))^{-1/\alpha}.    
\end{split}
\end{equation*}

For any $\varepsilon > 0$, there exists $\delta(\varepsilon) > 0$
such that for $x \in (0,\delta(\varepsilon))$ we have 
$- ( 1 + \varepsilon) x \leq \log ( 1 -x) \leq - x$. Since 
$f_n(s) \geq f_1(s) \geq s$, there exists $s(\varepsilon) \in (0,1)$
such that $g(s(\varepsilon)) \geq 1 - \delta(\varepsilon)$. Then 
for all $n = 0,1,\ldots$, and $s \in (s(\varepsilon),1)$
\begin{equation} \label{eq:gf-aux1}
- ( 1 + \varepsilon) ( 1 - g(f_n(s))) \leq \log g(f_n(s)) \leq 
- ( 1 - g(f_n(s))).
\end{equation}

\noindent \textbf{Case (B1):}
First, we suppose that the immigration has finite mean, 
$\E B = g'(1) < \infty$.
There exists $y_\varepsilon \in (0,1)$ such that 
$(1-y) (1-\varepsilon) g'(1) \leq 1- g(y) \leq (1-y) g'(1)$
for $y \in (y_\varepsilon, 1)$, we obtain from \eqref{eq:gf-aux1}
and Corollary \ref{cor:X-inf} that for 
$s \geq \max\{ y_\varepsilon, s_\varepsilon \}$
\begin{equation*} \label{eq:varphi-asy-aux1}
-(1 + \varepsilon) g'(1) \sum_{n=0}^\infty ( 1 - f_n(s)) \leq 
\log \varphi(s) \leq 
-(1 - \varepsilon) g'(1) \sum_{n=0}^\infty ( 1 - f_n(s)).
\end{equation*}
As $\varepsilon > 0$ is arbitrarily small, we obtain that 
\begin{equation} \label{eq:varphi-asy}
-\log \varphi(s) \sim 
g'(1) \sum_{n=0}^\infty ( 1 - f_n(s)) \quad 
\text{as } s \uparrow 1.
\end{equation}

Fix $N > 0$. Then, by \eqref{eq:m-asy0}, if $1-s$ is small enough
\begin{equation} \label{eq:phi-sum-aux1}
\sum_{n=0}^{\lfloor m(s)/N \rfloor} (1 - f_n(s))
\leq \frac{m(s)}{N} (1- s) \leq 
\frac{2 \alpha^{-1/\alpha} }{N} m(s)^{1-1/\alpha} \ell_1^{\#}(m(s))^{-1/\alpha},
\end{equation}
and as $s \uparrow 1$
\begin{equation} \label{eq:phi-sum-aux2}
\begin{split}
& \sum_{n=\lfloor N m(s) \rfloor}^\infty (n+m(s))^{-1/\alpha} 
\ell_1^\#(n+m(s))^{-1/\alpha} \\ 
& \leq \sum_{n=\lfloor N m(s) \rfloor}^\infty n^{-1/\alpha} \, 
\ell_1^\#(n)^{-1/\alpha}   \\
& \sim (N m(s))^{1-1/\alpha}\,\frac{1}{\tfrac{1}{\alpha} - 1} \, 
\ell_1^\#(N m(s))^{-1/\alpha} \\
& \sim \frac{\alpha}{1 - \alpha} N^{1 - 1/\alpha} \, m(s)^{1-1/\alpha} \, 
\ell_1^{\#}(m(s))^{-1/\alpha}.
\end{split}
\end{equation}

For the main contribution, by the uniform convergence theorem
for slowly varying functions
\begin{equation} \label{eq:phi-sum-aux3}
\begin{split}
& \sum_{n=\lfloor m(s)/N \rfloor}^{\lfloor N m(s) \rfloor} 
(n+m(s))^{-1/\alpha} \ell_1^\#(n+m(s))^{-1/\alpha} \\
&= m(s)^{1-\frac{1}{\alpha}} 
\sum_{n=\lfloor  \frac{m(s)}{N} \rfloor}^{\lfloor N m(s) \rfloor} 
\frac{1}{m(s)} \Big(\tfrac{n}{m(s)} + 1\Big)^{-1/\alpha} 
\ell_1^\#\!\big(m(s)(\tfrac{n}{m(s)} + 1)\big)^{-1/\alpha} \\
&\sim m(s)^{1-1/\alpha} \ell_1^\#(m(s))^{-1/\alpha} 
\int_{N^{-1}}^N (1 + y)^{-1/\alpha} \, \mathrm{d}y.
\end{split}
\end{equation}
Hence, letting $N \to \infty$, by \eqref{eq:varphi-asy},
\eqref{eq:phi-sum-aux1}, \eqref{eq:phi-sum-aux2},
and \eqref{eq:phi-sum-aux3} we obtain
\begin{equation} \label{eq:varphi-asy2}
\begin{split}
- \log \varphi(s) & \sim g'(1) \alpha^{-1/\alpha} \frac{\alpha}{1-\alpha}
m(s)^{1- 1/\alpha} \ell_1^\#(m(s))^{-1/\alpha}. 
\end{split}
\end{equation}
Finally, using \eqref{eq:m-asy} and that 
$\ell_1^{\#}(x \ell_1(x)) \sim \ell_1(x)^{-1}$ we have
\[
m(s)^{1-1/\alpha} \ell_1^{\#}(m(s))^{-1/\alpha} 
\sim \alpha^{1/\alpha -1} (1-s)^{1-\alpha} \ell_1((1-s)^{-\alpha}).
\]
Substituting back into \eqref{eq:varphi-asy2}
\[
- \log \varphi(s) 
\sim \frac{g'(1)}{1-\alpha}
(1-s)^{1-\alpha} \frac{1}{\ell_A(1/(1-s))}, \quad s \uparrow 1.
\]
Since $-\log \varphi(s) \sim 1- \varphi(s)$ as $s \uparrow 1$, the 
statement follows from Lemma \ref{lemma:tail-genfunc}.
\medskip 

\noindent \textbf{Case (B2):}
Now consider the case, when the immigration satisfies 
\eqref{eq:g-ass} with $\beta \in (\alpha,1)$, and a slowly 
varying $\ell_B$. 
Similarly, as before
\[
\begin{split}
& - \log \varphi(s) 
\sim \sum_{n=0}^\infty (1-g(f_n(s))) \\
& \sim \sum_{n=0}^\infty (1-f_n(s))^\beta 
\, \ell_B\!\big((1-f_n(s))^{-1}\big) \\
& \sim \sum_{n=0}^\infty (1-f_{n+m(s)}(0))^\beta 
\, \ell_B\!\big((1-f_{n+m(s)}(0))^{-1}\big) \\
& \sim \sum_{n=0}^\infty 
\big( \alpha (n+m(s)) \, \ell_1^{\#}(n+m(s))\big)^{-\frac{\beta}{\alpha}}
\ell_B \left( ((n+m(s)) {\ell_1^\#(n+m(s))})^{\frac{1}{\alpha}}\right) \\
&= \alpha^{-\beta/\alpha} 
\sum_{n=0}^\infty (n+m(s))^{-\beta/\alpha} \, \ell_2(n+m(s)),
\end{split}
\]
where the asymptotic equality holds as $s \uparrow 1$, and 
\[
\ell_2(x) = (\ell_1^\#(x))^{-\beta/\alpha} 
\ell_B\left( (x \ell_1^{\#}(x))^{1/\alpha} \right).
\]
This is analogous to the previous case, with $\alpha$ replaced by $\alpha/\beta$
and $(\ell_1^{\#})^{-1/\alpha}$ by $\ell_2$. 
Therefore, the same computation implies
\begin{equation} \label{eq:varphi2-asy}
-\log \varphi(s)  \sim \alpha^{-\beta/\alpha}
\frac{\alpha}{\beta - \alpha} \, m(s)^{1-\beta/\alpha} \, \ell_2(m(s)), 
\quad s \uparrow 1.
\end{equation}
Noting that $\ell_2(x \ell_1(x)) \sim \ell_1(x)^{\beta/\alpha} \ell_B(x^{1/\alpha})$,
using \eqref{eq:m-asy} we have
\[
m(s)^{1-\beta/\alpha} \, \ell_2(m(s))
\sim \alpha^{\beta/\alpha -1} (1-s)^{\beta -\alpha}
\ell_1((1-s)^{-\alpha}) \ell_B ((1-s)^{-1}).
\]
Substituting into \eqref{eq:varphi2-asy}
\[
-\log \varphi(s)  \sim 
\frac{1}{\beta - \alpha} (1-s)^{\beta -\alpha}
\frac{\ell_B ((1-s)^{-1})}{\ell_A((1-s)^{-1}) }, \quad s \uparrow 1.
\]
The result follows again from Lemma \ref{lemma:tail-genfunc}.

\subsection{Proofs of Theorem \ref{thm:tail-proc}}

The regular variation property implies that
$\mathcal{L}(X_0 / x | X_0 > x) \stackrel{\mathcal{D}}{\rightarrow} U_0$.
As $x_0 \to \infty$ by the law of large numbers a.s.
\[
\frac{1}{x_0} \left( \sum_{i=1}^{x_0} A_i + B_i \right)
\longrightarrow 1.
\]
Therefore, as $x \to \infty$
\[
\mathcal{L}\left( \frac{X_1}{X_0} \Big| X_0 > x \right) \longrightarrow 
\delta_1,
\]
the latter being the degenerate distribution at 1.
The statement follows by an induction argument.

\subsection{Proof of Theorem \ref{thm:sum}}

\begin{proof}[Proof of Lemma \ref{lemma:T-tail}]
Rewriting the equation $h(s) = s f(h(s))$, and using \eqref{eq:f-ass} 
we obtain
\[
\frac{(1-h(s))^{1+\alpha}}{1-s} = \frac{h(s)}{s} \frac{1}{\ell_A(1/(1-h(s)))}.
\]
Thus, as $s \uparrow 1$
\begin{equation*} \label{eq:h-asy}
(1- h(s))^{1+\alpha} \ell_A(1/(1-h(s))) \sim 1-s.
\end{equation*}
Then with $x = (1 - h(s))^{-1}$
\[
x \ell_{A,1}(x) \sim (1-s)^{-1/(1+\alpha)},
\]
thus \eqref{eq:h-form} follows from the definition of the de Bruijn conjugate.
The tail asymptotic \eqref{eq:T-tail} follows from 
Lemma \ref{lemma:tail-genfunc}.
\end{proof}

Before turning to the proof of Theorem \ref{thm:sum} we 
need a result on the tail behavior 
of random sums, when both the number of summands and the summands are heavy-tailed with infinite mean. Part (i) is Proposition B.2.5 and part (ii) is Lemma B.2.7 in \cite{MikoschWinten}.

\begin{lemma} \label{lemma:rsum-tail}
Let $Y, Y_1, \ldots$ be iid nonnegative random variables, and 
independent of $Y$'s, $\tau$ a nonnegative integer-valued random 
variable. 
\begin{itemize}
\item[(i)] 
Assume that $\p ( Y > x) \in \mathcal{RV}_{-\nu}$, 
for some $\nu \in (0,1]$, and $\E \tau < \infty$. If $\nu = 1$ 
further assume that $\p ( \tau > x) = o( \p ( Y > x))$, as 
$x  \to \infty$. Then as $x \to \infty$
\[
\p \left( \sum_{i=1}^\tau Y_i > x \right) \sim \E \tau \, \p ( Y > x).
\]

\item[(ii)] 
Assume that for some $\nu \in (0,1)$, and slowly varying $\ell_Y$
\begin{equation*} \label{eq:Y-ass}
\p ( Y > x) = \frac{\ell_Y(x)}{\Gamma(1-\nu) x^\nu},
\end{equation*}
and  for some $\mu \in (0,1)$ and slowly varying $\ell_\tau$
\begin{equation*} \label{eq:tau-ass}
\p (\tau > x) = \frac{\ell_\tau(x)}{\Gamma(1-\mu) x^\mu}.
\end{equation*}
Then
\[
\p \left( \sum_{i=1}^\tau Y_i > x \right) \sim 
\frac{1}{\Gamma(1- \mu \nu)}
\frac{\ell_Y(x)^\mu \ell_\tau(x^\nu /\ell_Y(x))}{x^{\mu \nu}}.
\]
\end{itemize}
\end{lemma}

We note that in (i) the extra condition $\p ( \tau > x) / \p ( Y > x) \to 0$
for $\nu = 1$ is very weak, since $x \p ( \tau > x) \to 0$ always holds 
by $\E \tau < \infty$. In general, the extra condition is needed, even 
if $\E Y = \infty$ is assumed. It is clear from the proof that similar statement holds in (ii) allowing $\mu = 1$ and $\nu = 1$ if the appearing slowly varying function belongs to the 
de Haan class, see the remark after Lemma \ref{lemma:T-tail}.
For the sake of readability, we decided to exclude this case.

\begin{proof}[Proof of Theorem \ref{thm:sum}]
Write 
\[
\begin{split}
X_n &  = B_n + \theta_{n} \circ B_{n-1} + 
\ldots + \theta_n \circ \ldots  \circ \theta_2 \circ B_1 + 
\theta_n \circ \ldots \circ \theta_1 \circ X_0 \\
& =: \sum_{i=1}^n \Theta_{i+1,n} \circ B_i + \Theta_{1,n} \circ X_0.
\end{split}
\]
We can decompose the partial sum as
\[
\begin{split}
S_n & = \sum_{i=1}^n X_i 
= \sum_{i=1}^n \left( \sum_{j=1}^i \Theta_{j+1,i} \circ B_j + 
\Theta_{1,i} \circ X_0 \right) \\
& = \sum_{j=1}^n \sum_{i=j}^n \Theta_{j+1,i} \circ B_j 
+ \sum_{i=1}^n \Theta_{1,i} \circ X_0
\\
& = \sum_{j=1}^n \sum_{i=j}^\infty \Theta_{j+1,i} \circ B_j
- \left( \sum_{j=1}^n \sum_{i=n+1}^\infty \Theta_{j+1,i} \circ B_j
+ \sum_{i=n+1}^\infty \Theta_{1,i} \circ X_0 \right) \\
& \quad + \sum_{i=1}^\infty \Theta_{1,i} \circ X_0 \\
& = S_{1,n} - S_{2,n} + S_{3}.
\end{split}
\]
Recalling the notation from \eqref{eq:T-def}, let 
$T = \sum_{i=0}^\infty Z_i$ denote the total population
in a critical Galton--Watson process $(Z_n)$, and let 
$T_1, T_2, \ldots$ be iid copies of $T$. Then
\[
S_{3} \stackrel{\mathcal{D}}{=} \sum_{i=1}^{X_0} (T_i - 1).
\]
Note that generation 0 is not included, that is the reason for the $-1$'s above. Furthermore,
\[
\begin{split}
S_{2,n} & = 
\sum_{j=1}^n \sum_{i=n+1}^\infty \Theta_{j+1,i} \circ B_j
+ \sum_{i=n+1}^\infty \Theta_{1,i} \circ X_0 \\
& = \sum_{i=n+1}^\infty 
\left( \sum_{j=1}^n \Theta_{j+1,i} \circ B_j + 
\Theta_{1,i} \circ X_0 \right)\\
& = \sum_{i=n+1}^\infty \Theta_{n+1,i} \circ 
\left( \sum_{j=1}^n \Theta_{j+1,n} \circ B_j 
+ \Theta_{1,n} \circ X_0 \right) \\
& = \sum_{i=n+1}^\infty \Theta_{n+1,i} \circ X_n
\stackrel{\mathcal{D}}{=} \sum_{k=1}^{X_{\infty}} (T_k - 1).
\end{split}
\]
That is, $S_{2,n} + S_{3} = O_\p ( 1)$.
Finally, $S_{1,n}$ is the sum of $n$ iid random variables
\[
S_{1,n} = \sum_{j=1}^n U_j,
\]
where $U, U_1, U_2,\ldots$ are iid, $U = \sum_{j=1}^B T_j$. 
Using Lemmas \ref{lemma:T-tail} and \ref{lemma:rsum-tail}
we can determine the tail behavior of $U$. 

If (B1) holds, then by Lemmas \ref{lemma:T-tail} and 
\ref{lemma:rsum-tail} (i) with $\tau = B$
\[
\p ( U > x) \sim \E B \, 
x^{-1/(1+\alpha)} \left( \ell_{A,1}^{\#}(x^{1/(1+\alpha)})
\Gamma(\alpha/(1+\alpha)) \right)^{-1}.
\]
While if (B2) holds, by 
Lemmas \ref{lemma:T-tail} and \ref{lemma:rsum-tail} (ii)
with $\tau = B$, $\nu = (1 + \alpha)^{-1}$, $\mu = \beta$,
\[
\p ( U > x) \sim 
x^{-\frac{\beta}{1+\alpha}} 
\frac{\ell_B\left( x^{{1}/{(1+\alpha)}} \ell_{A,1}^{\#}(x^{1/(1+\alpha)})
\right)}
{\Gamma \left( 1 - \frac{\beta}{1+\alpha} \right)
\ell_{A,1}^{\#}\left( x^{{1}/{(1+\alpha)}} \right)^{\beta}}.
\]
In both cases $U$ belongs to the domain of attraction of 
an $\eta$-stable law, and the result follows. 
\end{proof}


\end{document}